\documentclass[11pt]{amsproc}

\usepackage{mathtext}
\usepackage[cp1251]{inputenc}
\usepackage[T2A]{fontenc}

\usepackage[dvips]{graphicx}
\usepackage{amsmath}
\usepackage{amssymb}
\usepackage{amsxtra}
\usepackage{latexsym}
\usepackage{ifthen}
\usepackage{mathrsfs}

\def\N{{{\Bbb N}}}
\def\Z{{{\Bbb Z}}}
\def\T{{{\Bbb T}}}
\def\R{{\Bbb R}}

\def\l{{\lambda }}
\def\a{{\alpha }}
\def\D{{\Delta }}

\def\a{{\alpha}}
\def\b{{\beta}}
\def\d{{\delta}}
\def\e{{\varepsilon}}
\def\s{{\sigma}}
\def\vp{{\varphi}}
\def\t{{\theta }}
\def\g{{\gamma }}
\def\w{{\omega }}

\def\){\right)}
\def\({\left(}

\def\sign{\operatorname{sign}}

\numberwithin{equation}{section}

\newtheorem{corollary}{Corollary}[section]
\newtheorem{lemma}{Lemma}[section]
\newtheorem{theorem}{Theorem}[section]
\newtheorem{proposition}{Proposition}[section]
\newtheorem{remark}{Remark}[section]

\def\R{\Bbb R}

\def\XXint#1#2#3{{\setbox0=\hbox{$#1{#2#3}{\int}$}
     \vcenter{\hbox{$#2#3$}}\kern-.5\wd0}}

\par

\bigskip
\bigskip

\begin{document}

\title[Inequalities in approximation theory]{Inequalities in approximation theory involving fractional smoothness in $L_p$, $0<p<1$$^1$}

\author[Yurii
Kolomoitsev]{Yurii
Kolomoitsev$^{\text{a},\text{b},*}$}

\author[Tetiana
Lomako]{Tetiana
Lomako$^{\text{a},\text{b}}$}

\thanks{$^\text{a}$Universit\"at zu L\"ubeck,
Institut f\"ur Mathematik,
Ratzeburger Allee 160,
23562 L\"ubeck}

\thanks{$^\text{b}$Institute of Applied Mathematics and Mechanics of NAS of Ukraine,
Dobrovol's'kogo str.~1, Slov’yans’k, Donetsk region, Ukraine, 84100}

\thanks{$^1$Supported by the project AFFMA that has received funding from the European Union's Horizon 2020 research and innovation
programme under the Marie Sklodowska-Curie grant agreement No 704030.}

\thanks{$^*$Corresponding author}

\thanks{E-mail address: kolomoitsev@math.uni-luebeck.de, kolomus1@mail.ru}

\date{\today}
\subjclass[2000]{42A10, 26A33, 41A17, 41A25, 41A28} \keywords{Best approximation, trigonometric polynomials, fractional moduli of smoothness, fractional derivatives, the spaces $L_p$, $0<p<1$}

\begin{abstract}
In the paper, we study inequalities for the best trigonometric approximations and fractional moduli of smoothness involving the Weyl and Liouville-Gr\"unwald derivatives in $L_p$, $0<p<1$. We extend known inequalities to the whole range of parameters of smoothness as well as obtain several new inequalities.
As an application, the direct and inverse theorems of approximation theory involving the modulus of smoothness $\w_\b(f^{(\a)},\d)_p$, where $f^{(\a)}$ is a fractional derivative of the function $f$, are derived. A description of the class of functions with the optimal rate of decrease of a fractional modulus of smoothness is given.
\end{abstract}

\maketitle

\section{Introduction}

Let $\T\cong[0,2\pi)$ be the torus. As usual, the space
$L_p(\T)$, $0<p<\infty$, consists of measurable complex
functions that are $2\pi$-periodic  and
$$
\Vert f\Vert_p=\bigg(\frac1{2\pi}\int_{\T}|f(x)|^p {d}x\bigg)^{\frac 1p}<\infty.
$$

%


Recall that if $f\in L_1(\mathbb{T})$ and $\a\in \R$, then the fractional Weyl derivative $f^{(\a)}$ (if $\a<0$, the fractional integral)  is defined by
$$
f^{(\a)}(x)\sim\sum_{k\in \Z\setminus \{0\}} (ik)^\alpha \widehat{f}_k
e^{ikx},\quad (ik)^{\alpha}=|k|^\a e^{i\frac{\alpha\pi}{2}\textrm{sign}k},
$$
where $\widehat{f}_k=\frac1{2\pi}\int_0^{2\pi} f(x)e^{-ikx}dx$ are the Fourier coefficients of $f$.
There are several approaches to define fractional derivatives (see~\cite{SKM}). In this paper, together with the fractional Weyl derivative,  we use the fractional derivative in the sense of $L_p$ (or the Liouville-Gr\"unwald derivative).
For $f\in L_p(\T)$, $0<p< \infty$,
we define the derivative of order $\a>0$  in the sense of~$L_p$ as a function $g\in L_p(\T)$ satisfying
\begin{equation}\label{eqProizvLp}
    \bigg\Vert \frac{\D_h^\a f}{h^\a}-g\bigg\Vert_p\to
    0\quad\text{as}\quad h\to 0.
\end{equation}
As usual,
$$
\Delta_\delta^\alpha f(x)=\sum_{\nu=0}^\infty\binom{\alpha}{\nu}(-1)^\nu
f(x-\nu\d)
$$
and
$
\binom{\alpha}{\nu}=\frac{\alpha (\alpha-1)\dots
(\alpha-\nu+1)}{\nu!},\quad \nu\ge1,\quad \binom{\alpha}{0}=1.
$

Note that in the above definition if $f\in L_p(\mathbb{T})$, $1\le p<\infty$, and the function $g$ exists,  then $g$ coincides  with the fractional Weyl derivative $f^{(\a)}$ (see~\cite{BW}). Because of this, we  denote $g=f^{(\a)}$ for all  $0<p<\infty$.


Let $\mathcal{T}_n$ be the set of all trigonometric
polynomials of order at most $n$. The best approximation of a function $f$ by polynomials $T\in \mathcal{T}_n$ is given by
$$
E_n(f)_p=\inf_{T\in \mathcal{T}_n}\Vert f-T\Vert_p.
$$
As usual, if
$\Vert f-T\Vert_p=E_n(f)_p$ and $T\in \mathcal{T}_n$, then  $T$ is called a polynomial of the best approximation of $f$ in $L_p(\T)$.

Recall several known inequalities for the best trigonometric approximation of a function $f\in L_p(\T)$, $1\le p<\infty$, and its fractional derivatives of order $\a>0$. We have
\begin{equation}\label{+E}
 E_n(f)_p\leq \frac{C(\a)}{n^{\a}}E_n(f^{(\a)})_p,
\end{equation}
\begin{equation}\label{eqthsec3.A.2+}
E_n(f^{(\a)})_p\leq C(\a)\(n^\a E_n(f)_p+\sum\limits_{\nu=n+1}^\infty \nu^{\a -1}E_\nu(f)_p\),
\end{equation}
\begin{equation}\label{eqthsec3.A.2-}
\|f^{(\a)}-T_n^{(\a)}\|_p\leq C(\a)E_n(f^{(\a)})_p,
\end{equation}
where $T_n\in\mathcal{T}_n$ is such that $\|f-T_n\|_p=E_n(f)_p$. Remark that inequality~\eqref{+E} can be found in~\cite{BDGS77} and~\cite[p.~95]{Tab_book}; inequality~\eqref{eqthsec3.A.2+}, which is a (weak) inverse inequality to~\eqref{+E}, was proved in~\cite[pp.~150--153]{Tab_book} (see also~\cite{ST}); inequality~\eqref{eqthsec3.A.2-}, which is related to the simultaneous approximation of a function and its derivatives,   was derived for the case $\a\in \N$ in~\cite{CzFr}  and for the case $\a>0$ in~\cite{Tab_84}.


%
%
%

Inequalities of type \eqref{+E}--\eqref{eqthsec3.A.2-} have been  also studied in the case $0<p<1$, mainly for the derivatives of integer order. In particular, Ivanov~\cite{I} proved that if $f\in L_p(\T)$, $0<p<1$, is such that $\sum_{\nu=1}^\infty \nu^{\a p-1}E_\nu(f)_p^p<\infty$ for some $\a\in \N$, then $f$ has the derivative $f^{(\a)}$ in the sense of $L_p$ and
\begin{equation}\label{eqthsec3.A.2++}
E_n(f^{(\a)})_p\leq C(p,\a)\(n^\a E_n(f)_p+\(\sum\limits_{\nu=n+1}^\infty \nu^{\a p-1}E_\nu(f)_p^p\)^{\frac{1}{p}}\)\,.
\end{equation}
For $\a\ge 1$ and $1/2<p<1$, such result was obtained by Taberski~\cite{Tab_79}.

Concerning inequalities~\eqref{+E} and~\eqref{eqthsec3.A.2-}, it is known that in $L_p(\T)$, $0<p<1$, these inequalities  are not valid in general (see~\cite{I}, \cite{Kop95}, and~\cite{Di}). In particular, from~\cite{Kop95}, it follows that
for every $C>0$, $B\in \R$, $0<p<1$, and $n\in \N$, there exists a function $f_0\in AC(\T)$ (absolutely continuous functions) such that
\begin{equation}\label{eqKop}
  E_n(f_0)_p>Cn^B \Vert f_0'\Vert_{p}.
\end{equation}

The first positive results related to inequalities~\eqref{+E} and~\eqref{eqthsec3.A.2-} have been recently obtained in~\cite{KoArh}.  In particular, it is proved that if $\a\in\N$ and a function $f$ is such that $f^{(\a-1)}\in AC(\T)$, then
\begin{equation}\label{eqthsec3.1.2++}
  E_n(f)_p\leq\frac{C(\a,p)}{n^\a}\(E_n(f^{(\a)})_p+\(\frac1{n^{1-p}}
\sum\limits_{\nu=n+1}^\infty \frac{E_\nu(f^{(\a)})_p^p}{\nu^p}\)^{\frac{1}{p}}\).
\end{equation}
It is also shown the sharpness of the form of this inequality in the sense that ${\nu^{-p}}{E_\nu(f^{(\a)})_p^p}$ cannot be replaced by ${\nu^{-p-\e}}{E_\nu(f^{(\a)})_p^p}$ for any $\e>0$.

As a rule, problems related to the smoothness of functions in $L_p$, $0<p<1$, are essentially differ from the corresponding ones in the spaces
$L_p$, $p\ge 1$. Especially this is the case of the derivatives of fractional order. For example, the Bernstein inequality for the fractional derivatives in the case $0<p<1$ has the following form (see~\cite{BL}):
\begin{equation*}
\sup_{T_n\in \mathcal{T}_n,\,\Vert T_n\Vert_p\le 1}\Vert T_n^{(\a)}\Vert_p\asymp \left\{%
\begin{array}{ll}
n^{\a}, & \hbox{$\a\in\mathbb{N}$ or $\a\not\in\mathbb{N}$ and $\a>\frac 1p-1$,} \\
n^{\frac1p-1}\log^\frac 1p n, & \hbox{$\a=\frac 1p-1\not\in\mathbb{N}$,} \\
n^{\frac1p-1}, & \hbox{$\a\not\in\mathbb{N}$ and $\a<\frac 1p-1$,}
\end{array}%
\right.
\end{equation*}
where $\asymp$ is a two-sided inequality with absolute constants independent of $n$.
On the other hand, in the classical case $p\ge 1$,  we have $\Vert T_n^{(\a)}\Vert_p\le C(\a)n^\a\Vert T_n\Vert_p$ for any $\a>0$ (see, e.g.,~\cite{Ci}, \cite[Ch.~4, \S~19]{SKM}).
Other interesting "pathological" properties related to the smoothness of functions in the spaces $L_p$, $0<p<1$, can be found, e.g., in~\cite{DHI}, \cite{DiTi07}, \cite{Kr}, \cite{peetre}, \cite{SKO75}.

Now, let us
consider counterparts of inequalities~\eqref{+E} and~\eqref{eqthsec3.A.2+} for fractional moduli of smoothness.  Recall that the fractional modulus of smoothness of order $\alpha>0$ for a function $f\in L_p(\mathbb{T})$
is given by
\begin{equation}\label{eqmod1}
    \omega_\alpha(f,h)_p=\sup_{|\delta|<h} \Vert \Delta_\delta^\alpha f\Vert_p.
\end{equation}
For the first time, the modulus~\eqref{eqmod1} appeared in 1970's (see~\cite[p.~788]{BDGS77} and~\cite{tabeR}). At present, fractional moduli of smoothness are extensively studied and have several important applications to sharp inequalities of different metrics and embedding theorems (see~\cite{K11}, \cite{RS3}, \cite{Ru17}, \cite{ST}, \cite{ST10}, \cite{Ti05}, see also the monograph~\cite{PST2016} and the literature therein).

It is known that for any $f\in L_p(\T)$, $1\le p<\infty$, and $\a,\b>0$, the following two inequalities are fulfilled:
\begin{equation}\label{eqnerm1-}
    \w_{\b+\a}\(f,\d\)_p\le C\d^\a\w_\b
    (f^{(\a)},\d)_p,
\end{equation}
\begin{equation}\label{eqthModFrD.1-}
    \w_{\b}(f^{(\a)},\d)_p\le C\(\int_0^\d
    \frac{\w_{\b+\a}(f,t)_p^\t}{t^{\a\t+1}}dt\)^\frac1\t,
\end{equation}
where $\t=\min(2,p)$ and the constant $C$ is independent of $f$ and $\d$.
Remark that inequality~\eqref{eqnerm1-} can be found, e.g., in~\cite{BDGS77}; inequality~\eqref{eqthModFrD.1-} is proved in~\cite{ST} (see also~\cite{johnen} and~\cite{DiTi07} for $\a,\b\in \N$).

It turns out that in the case $0<p<1$, inequalities~\eqref{eqnerm1-} and~\eqref{eqthModFrD.1-} have been studied only for integer parameters $\a$ and $\b$. At that, in this case ($\a,\b\in \N$ and $0<p<1$), the analogue of~\eqref{eqthModFrD.1-} is known and its form coincides with~\eqref{eqthModFrD.1-}   (see~\cite{DiTi07}). In contrast with~\eqref{eqthModFrD.1-}, inequality~\eqref{eqnerm1-} is not valid if $0<p<1$. As it was mentioned in~\cite[p.~188]{PePo},
''\emph{there is no upper estimate of $\w_{k}(f,\d)_p$ by $\w_{k-1}(f^\prime,\d)_p$ in the case $0<p<1$}''. However, the modulus $\w_{k}(f,\d)_p$  can be estimated from above by means of a certain integral expression related to~\eqref{eqthsec3.1.2++} with the modulus $\w_{k-1}(f^\prime,\nu^{-1})_p$ instead of the corresponding best approximation  (see~\cite{KoArh}). 

In this paper, we obtain analogous of inequalities~\eqref{eqthsec3.A.2-}, \eqref{eqthsec3.A.2++}, and~\eqref{eqthsec3.1.2++} as well as~\eqref{eqnerm1-} and~\eqref{eqthModFrD.1-} for any $f\in L_p(\T)$, $0<p<1$, and any admissible parameters $\a,\b>0$ (see Theorems~\ref{thsec3.1}--\ref{thModFrD10}).

As an application of inequalities of type~\eqref{eqthsec3.A.2++} and~\eqref{eqthsec3.1.2++}, we derive the direct and inverse theorems of the approximation theory involving the modulus of smoothness $\w_\b(f^{(\a)},\d)_p$ (see Theorems~\ref{thsec3.1++++}, \ref{thsec3.A++++}, and~\ref{-thsec3.A++++}). At the same time, corresponding analogues of~\eqref{eqthModFrD.1-} and~\eqref{eqthsec3.1.2++} are applied to describe the class of functions with the optimal rate of decrease of $\w_\b(f,\d)_p$ in the case $0<p<1$ (see Theorem~\ref{th2}). 

\section{Main results}

\subsection{Inequalities for the best approximation}
We start this section with the following counterpart of inequality~\eqref{+E}
in the case  $0<p<1$.

\begin{theorem}\label{thsec3.1}
{\it Let $0<p<1$, $\a>0$, and let $f$ be such that $f,f^{(\a)}\in L_1(\T)$.
Then for any $n\in \N$ we have
\begin{equation}\label{eqthsec3.1.2}
  E_n(f)_p\leq\frac{C}{n^\a}\(E_n(f^{(\a)})_p+\(\frac1{n^{1-p}}\sum\limits_{\nu=n+1}^\infty\frac{E_\nu(f^{(\a)})_p^p}{\nu^{p}}\)^{\frac{1}{p}}\)\,,
\end{equation}
where $C$ is a constant independent of $f$ and $n$.}
\end{theorem}

\begin{remark}\label{Rem_17}

\emph{1) Inequality~\eqref{eqKop} implies that Theorem~\ref{thsec3.1} is not valid without the second summand in right-hand side of~\eqref{eqthsec3.1.2}.}

\emph{2) In Theorem~\ref{thsec3.1}, the assumption $f,f^{(\a)}\in L_1(\T)$   cannot be replaced by the much weaker assumption of existence of the derivative $f^{(\a)}$ in the sense of $L_p(\T)$.
Indeed, let us consider the function $f(x)={\sign}\sin(x)$. We have that $f$ has the derivative in the sense of $L_p$, $0<p<1$, and $f'(x)=0$ a.e. Thus, in the case $\a=1$,  the right-hand side of~\eqref{eqthsec3.1.2} is zero while $E_n(f)_p>0$ for all $n\in\N$ which is impossible.
Moreover, it follows from the proof of Theorem~\ref{thsec3.1} that  the convergence of the series in~\eqref{eqthsec3.1.2} implies that $f^{(\a)}\in L_1(\T)$.}

\end{remark}

The next theorem gives an inverse inequality to (\ref{eqthsec3.1.2}).

\begin{theorem}\label{thsec3.A}
Let $f\in L_p(\mathbb{T})$, $0<p<1$, and let for some $\a\in \N\cup(1/p-1,\infty)$
\begin{equation}\label{eqthsec3.A.1}
\sum\limits_{\nu=1}^\infty \nu^{\a p-1}E_\nu(f)_p^p<\infty\,.
\end{equation}
Then $f$ has the derivative $f^{(\a)}$ in the sense of $L_p$ and for any $n\in\N$
\begin{equation}\label{eqthsec3.A.2}
\|f^{(\a)}-T_n^{(\a)}\|_p\leq C\(n^\a E_n(f)_p+\(\sum\limits_{\nu=n+1}^\infty \nu^{\a p-1}E_\nu(f)_p^p\)^{\frac{1}{p}}\)\,,
\end{equation}
where $T_n\in\mathcal{T}_n$ is such that $\|f-T_n\|_p=E_n(f)_p$ and $C$ is a constant independent of $f$ and $n$.
\end{theorem}

Remark that in the case $\a\in \N$,  Theorem~\ref{thsec3.A} was proved in~\cite{I} while the case $1/2<p<1$ and $\a\ge 1$ was considered in~\cite{Tab_79}.

Under additional restrictions on the function $f$ in Theorem~\ref{thsec3.A}, it is possible to obtain an~analogue of inequality~\eqref{eqthsec3.A.2} for any $\a>0$.

\begin{theorem}\label{propanal1}
 Let $0<p<1$, $\a>0$, and let $f$ be such that $f,f^{(\a)}\in L_1(\T)$.
Then for any $n\in \N$ we have
\begin{equation}\label{eqthsec3.A.2+A}
\|f^{(\a)}-T_n^{(\a)}\|_p\leq C\(\s_{\a,p}(n) E_n(f)_p+\(\sum\limits_{\nu=n+1}^\infty (\s_{\a,p}(\nu))^p\nu^{-1}E_\nu(f)_p^p\)^{\frac{1}{p}}\)\,,
\end{equation}
where
$T_n\in\mathcal{T}_n$ is such that $\|f-T_n\|_p=E_n(f)_p$,
\begin{equation*}
\s_{\a,p}(n)=\left\{%
\begin{array}{ll}
n^{\a}, & \hbox{$\a\in\mathbb{N}$ or $\a\not\in\mathbb{N}$ and $\a>\frac 1p-1$,} \\
n^{\frac1p-1}\log^\frac 1p (n+1), & \hbox{$\a=\frac 1p-1\not\in\mathbb{N}$,} \\
n^{\frac1p-1}, & \hbox{$\a\not\in\mathbb{N}$ and $\a<\frac 1p-1$,}  \\
\end{array}%
\right.
\end{equation*}
and $C$ is a constant independent of $f$ and $n$.
\end{theorem}

Combining Theorems~\ref{thsec3.1} and~\ref{propanal1}, we derive a positive result about the simultaneous approximation of functions and their derivatives in the spaces $L_p(\T)$, $0<p<1$.  

\begin{theorem}\label{thsec3.2}
{\it Let $0<p<1$, $\a>0$, and let $f$ be such that $f,f^{(\a)}\in L_1(\T)$. 
Then for any $n\in\N$ we have
\begin{equation}\label{eqthsec3.2.2}
  \|f^{(\a)}-T_n^{(\a)}\|_p\leq C\rho(n)\(E_n(f^{(\a)})_p+\(\frac1{n^{1-p}}\sum\limits_{\nu=n+1}^\infty\frac{E_\nu(f^{(\a)})_p^p}{\nu^{p}}\)^{\frac{1}{p}}\)\,,
\end{equation}
where $T_n\in\mathcal{T}_n$ is such that $\|f-T_n\|_p=E_n(f)_p$,
\begin{equation*}
\rho_{\a,p}(n)=\left\{%
\begin{array}{ll}
1, & \hbox{$\a\in\mathbb{N}$ or $\a\not\in\mathbb{N}$ and $\a>\frac 1p-1$,} \\
\log^\frac 1p (n+1), & \hbox{$\a=\frac 1p-1\not\in\mathbb{N}$,} \\
n^{\frac1p-1-\a}, & \hbox{$\a\not\in\mathbb{N}$ and $\a<\frac 1p-1$,}  \\
\end{array}%
\right.
\end{equation*}
and $C$ is a constant independent of $f$ and $n$.
}
\end{theorem}

Using Theorems~\ref{thsec3.1},~\ref{thsec3.A}, and~\ref{thsec3.2}, we get the following equivalences.

\begin{corollary}\label{corsec3.1E}
{\it Let $0<p<1$, $\g>{1}/{p}-1$, $\a\in \N\cup(1/p-1,\infty)$, and let $f$ be such that $f,f^{(\a)}\in L_1(\T)$. Then the following assertions are equivalent:

\medskip

$(i)$ $E_n(f)_p=\mathcal{O}(n^{-\a-\g})\,, \quad n\rightarrow
\infty\,,$

\medskip

$(ii)$ $E_n(f^{(\a)})_p=\mathcal{O}(n^{-\g})\,, \quad n\rightarrow
\infty\,,$

\medskip

$(iii)$ $\Vert f^{(\a)}-T_n^{(\a)}\Vert_p=\mathcal{O}(n^{-\g})\,, \quad
n\rightarrow \infty\,,$

\medskip

\noindent where $T_n\in\mathcal{T}_n$ is such that $\|f-T_n\|_p=E_n(f)_p$.}
\end{corollary}


\subsection{Inequalities for the moduli of smoothness}

In this subsection, similarly to the above considered case of the best approximation, we obtain  counterparts of~\eqref{eqnerm1-} and~\eqref{eqthModFrD.1-} for $0<p<1$.

\begin{theorem}\label{thDif}
{\it Let $0<p<1$, $\a>0$ and $\b\in \N\cup
(1/p-1,\infty)$ be such that $\a+\b\in\N\cup
(1/p-1,\infty)$, $r\in \N$, and $f,f^{(\a)}\in L_1(\T)$. Then for any $\d>0$
\begin{equation}\label{eqDif1}
  \w_{\b+\a}(f,\d)_p\le C\d^\a \(\w_\b(f^{(\a)},\d)_p+\(\d^{1-p}\int_0^\d
    \frac{\w_r(f^{(\a)},t)_p^p}{t^{2-p}}dt\)^\frac1p\),
\end{equation}
where $C$ is a constant independent of $f$ and
$\d$.
}
\end{theorem}

In particular, under the conditions of Theorem~\ref{thDif}, one has
$$
\w_{\b+\a}(f,\d)_p\le C\d^{\a+\frac1p-1} \(\int_0^\d
    \frac{\w_\b(f^{(\a)},t)_p^p}{t^{2-p}}dt\)^\frac1p.
$$

A converse result is given by the following theorem.

\begin{theorem}\label{thModFrD}
{\it Let $f\in L_p(\T)$, $0<p<1$, and $\a,\b\in \N\cup
(1/p-1,\infty)$. Then
\begin{equation}
\label{eqthModFrD.1}
    \w_{\b}(f^{(\a)},\d)_p\le C\(\int_0^\d
    \frac{\w_{\b+\a}(f,t)_p^p}{t^{p\a+1}}dt\)^\frac1p,
\end{equation}
where $C$ is some constant independent of $f$ and
$\d$. Inequality (\ref{eqthModFrD.1}) means that if the right-hand is finite, then there exists
 $f^{(\a)}$ in the sense
(\ref{eqProizvLp}), $f^{(\a)}\in L_p(\T)$, and
(\ref{eqthModFrD.1}) holds.}
\end{theorem}

Under additional restrictions on $f$, we obtain an analogue of inequality~\eqref{eqthModFrD.1} for any $\a>0$.

\begin{theorem}\label{thModFrD10}
{\it Let $0<p<1$, $\a>0$ and $\b\in \N\cup
(1/p-1,\infty)$ be such that $\a+\b\in\N\cup
(1/p-1,\infty)$, and $f,f^{(\a)}\in L_1(\T)$. Then for any $\d>0$
\begin{equation*}
    \w_{\b}(f^{(\a)},\d)_p\le C\(\int_0^\d
    \frac{\w_{\b+\a}(f,t)_p^p}{t}\s_{\a,p}\(\frac1t\)dt\)^\frac1p,
\end{equation*}
where $\s_{\a,p}(\cdot)$ is defined in Theorem~\ref{propanal1} and $C$ is some constant independent of $f$ and
$\d$.}
\end{theorem}

The next corollary easily follows from  Theorems~\ref{thDif} and~\ref{thModFrD}.

\begin{corollary}\label{corsec3.1E}
{\it Let $0<p<1$, $\g>{1}/{p}-1$, $\a,\b\in \N\cup(1/p-1,\infty)$, and let $f$ be such that $f,f^{(\a)}\in L_1(\T)$. Then the following assertions are equivalent:

\medskip

$(i)$ $\w_{\a+\b}(f,\d)_p=\mathcal{O}(\d^{\a+\g})\,, \quad \d\rightarrow 0\,,$

\medskip

$(ii)$ $\w_\b(f^{(\a)},\d)_p=\mathcal{O}(\d^{\g})\,, \quad \d\rightarrow 0\,.$
}
\end{corollary}

\subsection{The direct and inverse approximation theorems}
Let us recall two basic inequalities in approximation theory (the so-called direct and inverse approximation theorems).
\begin{proposition}\label{propCl} {\sc (See~\cite{RS3}.)}
Let $f\in L_p(\T)$, $0< p<1$, $\b\in \N\cup (1/p-1,\infty)$, and $n\in\N$. Then
\begin{equation}\label{eqJ}
  E_n(f)_p\le C\w_\b\(f,\frac 1n\)_p,
\end{equation}
\begin{equation}\label{eqB}
 \omega_\b\left(f,\frac 1n\right)_p\le\frac{C}{n^\b}
\(\sum_{\nu=0}^n(\nu+1)^{\b p-1}E_\nu(f)_p^p\)^\frac 1p,
\end{equation}
where $C$ is a constant independent of $n$ and $f$.
\end{proposition}

Remark that in the case $\b\in\N$  inequality~\eqref{eqJ}, which is also called the Jackson type inequality, was proved in~\cite{SO} (see also~\cite{SKO75} and~\cite{I}) and inequality~\eqref{eqB} was proved in~\cite{I} (see also~\cite{Tab_79} concerning the case $\b\ge 1$ and $1/2\le p<1$).

Using Theorem~\ref{thsec3.1}, it is not difficult to obtain the following extensions of inequality~\eqref{eqJ} involving fractional derivatives of the function $f$.

\begin{theorem}\label{thsec3.1++++}
{\it Let $0<p<1$, $\a>0$, $\b\in \N\cup(1/p-1,\infty)$, and let $f$ be such that $f,f^{(\a)}\in L_1(\T)$. 
Then for any $n\in \N$ we have
\begin{equation}\label{eqthsec3.1.2++++J}
  E_n(f)_p \leq\frac{C}{n^{\a+\frac1p-1}}\bigg(\int_0^{1/n}
    \frac{\w_\b(f^{(\a)},t)_p^p}{t^{2-p}}dt\bigg)^\frac1p\,,
\end{equation}
where $C$ is a constant independent of $f$ and $n$.}
\end{theorem}

Note that in the case $1\le p<\infty$, inequality~\eqref{eqthsec3.1.2++++J}  holds in the following form:
\begin{equation}\label{JJJ}
  E_n(f)_p\leq\frac{C}{n^\a}\w_\b\(f^{(\a)},\frac1n\)_p.
\end{equation}
Sometimes~\eqref{JJJ}  is called the second Jackson inequality
(see, e.g.,~\cite[p.~260]{timan}). Let us again emphasize that this inequality is not valid if $0<p<1$ (see~\eqref{eqKop}).

In the next theorem, using Theorem~\ref{thsec3.A} and~\eqref{eqB}, we obtain a converse inequality to~\eqref{eqthsec3.1.2++++J}.

\begin{theorem}\label{thsec3.A++++}
Let $f\in L_p(\mathbb{T})$, $0<p<1$, $\b\in \N\cup(1/p-1,\infty)$, and let for some $\a\in \N\cup(1/p-1,\infty)$
\begin{equation}\label{eqthsec3.A.1++++}
\sum\limits_{\nu=1}^\infty \nu^{\a p-1}E_\nu(f)_p^p<\infty\,.
\end{equation}
Then $f$ has the derivative $f^{(\a)}$ in the sense of $L_p$ and for any $n\in\N$
\begin{equation}\label{eqthsec3.1.2++++}
  \w_\b\(f^{(\a)},\frac1n\)_p\le C\(\frac1{n^{\b p}}\sum_{\nu=0}^n (\nu+1)^{(\a+\b)p-1}E_\nu(f)_p^p+\sum_{\nu=n+1}^\infty \nu^{\a p-1}E_\nu(f)_p^p\)^\frac1p\,,
\end{equation}
where $C$ is a constant independent of $f$ and $n$.
\end{theorem}

Remark that in the case $1\le p<\infty$, inequalities of type~\eqref{eqthsec3.1.2++++} can be found in~\cite{Tim} and \cite[p.~154]{Tab_book} (see also the general case in~\cite{ST}).

Similarly to Theorem~\ref{propanal1}, under additional restrictions on the function $f$, we obtain the following extension of Theorem~\ref{thsec3.A++++} to the case $\a\le 1/p-1$.

\begin{theorem}\label{-thsec3.A++++}
  Let $0<p<1$, $\a>0$,  $\b\in \N\cup(1/p-1,\infty)$, and let $f$ be such that $f,f^{(\a)}\in L_1(\T)$.
Then
\begin{equation*}
\begin{split}
    \w_\b\(f^{(\a)},\frac1n\)_p\le C\Bigg(\frac1{n^{\b p}}\sum_{\nu=0}^n (\s_{\a,p}(\nu+1))^p&(\nu+1)^{\b p-1}E_\nu(f)_p^p\\
&+\sum_{\nu=n+1}^\infty (\s_{\a,p}(\nu))^p\nu^{-1}E_\nu(f)_p^p\Bigg)^\frac1p\,,
\end{split}
\end{equation*}
where $\s_{\a,p}(\cdot)$ is defined in Theorem~\ref{propanal1}
and $C$ is a constant independent of $f$ and $n$.
\end{theorem}

Recall that the assertions of Proposition~\ref{propCl} imply that for $\b\in \N\cup (1/p-1,\infty)$ and $0<\g<\b$ the condition $\w_\b(f,\d)_p=\mathcal{O}(\d^\g)$ is equivalent to $E_n(f)_p=\mathcal{O}(n^{-\g})$. Combining Theorems~\ref{thsec3.1++++} and~\ref{thsec3.A++++}, we obtain an analogue of this equivalence involving the fractional derivative of a function $f$.

\begin{corollary}\label{corsec3.1EMOD}
{\it Let $0<p<1$, $\a,\b\in \N\cup(1/p-1,\infty)$, $1/p-1<\g<\b$, and let $f$ be such that $f,f^{(\a)}\in L_1(\T)$. Then the following assertions are equivalent:

\medskip

$(i)$ $E_n(f)_p=\mathcal{O}(n^{-\a-\g})\,, \quad n\rightarrow
\infty\,,$

\medskip

$(ii)$ $\w_\b(f^{(\a)},\d)_p=\mathcal{O}(\d^{\g})\,, \quad \d\rightarrow 0\,.$
}
\end{corollary}

\subsection{On decreasing of the fractional modulus of smoothness}

The following inequality plays a crucial role in the proofs of the main results of this paper:
\begin{equation}\label{eqM2}
    \omega_\b(f,\l \d)_p\le C(p,\b)(1+\l)^{\b+\frac 1{p_1}-1}\omega_\b(f,\d)_p,\quad
    \l,\,\d>0,
\end{equation}
where $\b\in \N\cup (1/{p_1}-1,\infty)$, $p_1=\min(p,\,1)$ (see
\cite{BDGS77} for the case $p\ge 1$ and \cite{RS3} for the case $0<p<1$).
Inequality $(\ref{eqM2})$ implies that the optimal rate of decrease of the modulus of smoothness $\omega_\b(f,h)_p$ as $h\to 0$
is $\mathcal{O}(h^{\b+1/{p_1}-1})$, that is if
$\omega_\b(f,h)_p=o(h^{\b+1/{p_1}-1})$, then $f\equiv \textrm{const}$ (see also Proposition 5.1 in~\cite{BW}). It arises a natural question  about characterization of the class of functions $f\in L_p(\T)$, $0<p<\infty$, such that
\begin{equation*}
  \omega_\b(f,h)_p=\mathcal{O}(h^{\b+1/{p_1}-1})\quad \text{as}\quad h\to 0.
\end{equation*}

The first characterization of this class was derived by Hardy and Littlewood~\cite{hl} in the case $\b=1$ and $1\le p<\infty$.
Their result was extended to the moduli of smoothness of integral order in~\cite{br2}
(see also \cite[Ch.~1, \S 9]{DeLo} and \cite[Theorem~4.6.14]{TB}) and to the fractional moduli of smoothness in~\cite{BW}.
In particular, the following proposition was proved by Butzer and Westphal~\cite{BW}.

\begin{proposition}\label{propBW}
Let $f\in L_p(\T)$, $1\le p< \infty$, and $\b>0$. Then $\omega_\beta(f,h)_p=\mathcal{O}(h^{\beta})$ if and only if $f$
 can be corrected on a set of measure zero to be a function $g$ such that
 $g^{(\b)}\in L_p(\T)$ for $1< p<\infty$  and $g^{(\b-1)}\in \textrm{BV}(\T)$ (functions of bounded variation on $\T$) for $p=1$.
\end{proposition}


In the spaces $L_p(\T)$, $0< p<1$, the class of functions with the optimal rate of decrease of the modulus of smoothness has different nature.
Indeed, it is easy to see that for any step function $f$ one has
$\omega_1(f,h)_p=\mathcal{O}(h^{{1}/{p}})$. A complete description of such functions was obtained by
Krotov~\cite{Kr}.


\begin{proposition}\label{propKrot} {\sc (See~\cite{Kr}.)}
Let $f\in L_p(\T)$, $0<p<1$. Then $\omega_1(f,h)_p=\mathcal{O}(h^{{1}/{p}})$ if and only if $f$
can be corrected on a set of measure zero to be a function $g$ such that $g(x)=d_0+\sum_{x_k<x}d_k$, where $\sum_k
|d_k|^p<\infty$ and $\{x_k\}$ is a set of pairwise distinct point from $[0,2\pi)$.
\end{proposition}

See also in~\cite{br3} and~\cite{kolmod}  analogues of Proposition~\ref{propKrot} with the moduli of smoothness of arbitrary integer order.


It is worth mentioning the following unusual property of the modulus of continuity: \emph{if $f\in AC(\T)$ (absolutely continuous functions on $\mathbb{T}$) and
$$
\omega_1(f,h)_p=o(h)\,\,\,\,\,\textrm{as}\,\,\,\,\,h\to 0
$$
for some $0< p< 1$, then $f\equiv\textrm{\rm const}$ a.e. on $\mathbb{T}$} (see Lemma 1.5 in~\cite{SKO75}).
%

\smallskip

Using Theorem~\ref{thModFrD}, it is easy to extend this property to the moduli of smoothness of fractional order. In particular, we have the following result:

\begin{proposition}\label{proposition}
Let $0<p<1$, $\b\in \N\cup (1/p-1,\infty)$, $f^{(\b-1)}\in AC(\T)$, and
$$
\w_\b(f,\d)_p=o(\d^\b),\quad \d\to 0,
$$
then $f\equiv\textrm{\rm const}$ a.e. on $\mathbb{T}$.
  \end{proposition}


In the next theorem, we generalize Proposition~\ref{propBW} to the case  $0<p<1$ and Proposition~\ref{propKrot} to the fractional moduli of smoothness of arbitrary order $\b\in \N\cup (1/p-1,\infty)$.

\begin{theorem}\label{th2}
{\it
Let $f\in L_p(\T)$, $0<p<1$, and
$\b\in \N\cup(1/p-1,\infty)$. Then the following assertions are equivalent:
\begin{enumerate}

\item[$(i)$] $\omega_\b(f,h)_p=\mathcal{O}(h^{\b+1/p-1})$ as $h\to 0$,

\smallskip

\item[$(ii)$]  $f\in L_1(\T)$ and it can be corrected on a set of measure zero to be a function $g$ such that $g^{(\b-1)}(x)=d_0+\sum_{x_k<x}d_k$, where $\sum_k
|d_k|^p<\infty$ and $\{x_k\}$ is a set of pairwise distinct point from $[0,2\pi)$.

\end{enumerate}
}
\end{theorem}


\section{Auxiliary results}

\subsection{Properties of the fractional moduli of smoothness}
First of all, we note that in the case $0<p<1$, considering the fractional derivatives  in the sense of $L_p$ and the corresponding  moduli of smoothness, we restrict ourselves to the parameter $\a$ belonging to the set $\N\cup (1/p-1,\infty)$. This restriction is natural since for $\a\in\N\cup (1/p-1,\infty)$ we always have
\begin{equation}\label{restr}
  \Vert \D_\d^\a f\Vert_{p}^p\le\sum_{\nu=0}^\infty \Big|\binom{\a}{\nu}\Big|^p\Vert f\Vert_p^p\le C(\a,p)\Vert f\Vert_p^p.
\end{equation}
The last inequality follows from the fact that $|\binom{\a}{\nu}|=\mathcal{O}({\nu^{-\a-1}})$ as $\nu\to\infty$ (see, e.g.,~\cite[Ch.~1,\,\S 1]{SKM}).

Let us recall two basic properties of the fractional moduli of smoothness.
For $f\in L_p(\T)$, $0<p\le \infty$, and $\a \in (1/p_1-1,\infty)\cup \N$, we have
\begin{equation}\label{eqM0}
       \omega_\a(f+g,\d)_p^{p_1}\le \omega_\a(f,\d)_p^{p_1}+\omega_\a(g,\d)_p^{p_1},\quad \d>0,
\end{equation}
\begin{equation}\label{eqM1}
       \omega_\a(f,\d)_p\le C(p,\a)\Vert f\Vert_p,\quad \d>0,
\end{equation}
where $p_1=\min(p,\,1)$. Inequality~\eqref{eqM0} is obvious while inequality~\eqref{eqM1} can be derived from~\eqref{restr}.

It is well known (see~\cite{BDGS77}) that if $1\le p\le\infty$, the
modulus of smoothness is equivalent to the $K$-functional given by
$$
 {K}_{\a}(f,\d)_{p}=\inf_{g^{(\a)}\in L_p(\T)} \(\|f-g\|_{p}+ \delta^\alpha \| g^{(\a)}\|_{p}\),
$$
that is,
\begin{equation*}
\omega_\a(f,\delta)_{p}
\asymp {K}_{\a}(f,\d)_{p},\quad \d>0.
\end{equation*}

This equivalence fails for  $0<p<1$ since $K_\a(f,\d)_{p}\equiv 0$ (see \cite{DHI}). A
suitable substitute for the $K$-functional for $p<1$ is the
realization concept given by
\begin{equation*}
\mathcal{R}_{\a}(f,\d)_{p}=\inf_{T\in\mathcal{T}_{[1/\d]}}\(\Vert
f-T\Vert_{p}+\d^{\a}\Vert T^{(\a)}\Vert_{p}\).
\end{equation*}

Let us recall some properties of the realization $\mathcal{R}_{\a}(f,\d)_{p}$.

\begin{lemma}\label{LemKfunc}
  Let $f\in L_p(\T)$, $0<p\le\infty$, and $\a\in \N\cup (1/p_1-1,\infty)$. Then
$$
\mathcal{R}_{\a}(f,\d)_{p}\asymp \w_\a(f,\d)_p,\quad \d>0,
$$
where $\asymp$ is a two-sided inequality with absolute constants independent of $f$ and $\d$.
\end{lemma}

Remark that in the case $\a\in \N$, Lemma~\ref{LemKfunc} was proved in~\cite{DHI}; the case $\a>1/p_1-1$ was considered in~\cite{K11} and~\cite{ST}.

The next lemma gives an analogue of inequality~\eqref{eqM2} for the realizations of $K$-functional.

\begin{lemma}\label{LemKfunc+} {\sc (See \cite[Theorem 4.22]{run}, \cite{RS3}).}
  Let $f\in L_p(\T)$, $0<p\le\infty$, and $\a>0$. Then
$$
\mathcal{R}_\a(f,\l \d)_p\le C(1+\l)^{\a+\frac 1{p_1}-1}\mathcal{R}_\a(f,\d)_p,\quad
    \l,\,\d>0,
$$
where $C$ is a constant, which depends only on  $p$ and $\a$.
\end{lemma}

Note that in above inequality in contrast with~\eqref{eqM2}, we do not assume that $\a>1/p-1$ in the case $0<p<1$.

\subsection{Inequalities for  trigonometric polynomials}

We need the following three important results  for trigonometric polynomials in $L_p$.
The first one is the Nikolskii--Stechkin type inequality (see~\cite{DHI} for the case $\a\in \N$ and~\cite{K07} for the case $\a>0$).

\begin{lemma}\label{lem1}
Let $0<p<1$, $n\in\N$, $0<h\le \pi/n$, and $\a>0$. Then for any trigonometric polynomial $T_n \in \mathcal{T}_n$, we have
\begin{equation*}
  h^\a\Vert T_n^{(\a)}\Vert_p\asymp \Vert \D_h^\a T_n\Vert_p,
\end{equation*}
where $\asymp$ is a two-sided inequality with absolute constants independent of $T_n$ and $h$.
Moreover, if $\a\in \N\cup (1/p-1,\infty)$ and $T_n$ is a polynomial of the best approximation of $f\in L_p(\T)$, then
\begin{equation*}
\Vert \D_h^\a T_n\Vert_p\le C\w_\a\(f,\frac1n\)_p,
\end{equation*}
where $C$ is a constant independent of $T_n$, $h$, and $f$.
\end{lemma}

The second result is the well-known Nikolskii inequality of different metrics (see, e.g.,~\cite[p.~133]{Nik}   and~\cite[Ch. 4, \S~2]{DeLo}).
\begin{lemma}\label{lemNikpq}
{\it Let $0<p<q\le\infty$. Then for any  $T_n\in \mathcal{T}_n$, $n\in \N$, one has
\begin{equation*}
  \Vert T_n \Vert_q\le C n^{\frac1p-\frac1q}\Vert T_n\Vert_p,
\end{equation*}
where $C$ is a constant independent of $T_n$.}
\end{lemma}

The third result is the Bernstein type inequality involving the Weyl fractional derivative  (see~\cite{BL}).

\begin{lemma}\label{lemBLf} {\it Let $0<p<1$. Then
\begin{equation*}
\sup_{T_n\in \mathcal{T}_n,\,\Vert T_n\Vert_p\le 1}\Vert T_n^{(\a)}\Vert_p\asymp \left\{%
\begin{array}{ll}
n^{\a}, & \hbox{$\a\in\mathbb{Z}_+$ or $\a\not\in\mathbb{Z}_+$ and $\a>\frac 1p-1$,} \\
n^{\frac1p-1}\log^\frac 1p n, & \hbox{$\a=\frac 1p-1\not\in\mathbb{Z}_+$,} \\
n^{\frac1p-1}, & \hbox{$\a\not\in\mathbb{Z}_+$ and $\a<\frac 1p-1$,}  \\
\end{array}%
\right.
\end{equation*}
where $\asymp$ is a two-sided inequality with absolute constants independent of $n$.}
\end{lemma}

\subsection{Approximation of a function and its derivatives}

In the spaces $L_p$ with  $p\ge 1$, the following fact is well-known:
{\it if a sequence of functions
$\{\varphi_n\}_{n=1}^\infty\subset L_p$ is such that
$\varphi_n^{(r-1)}\in AC$, $n\in \mathbb{N}$, and for some $f,g\in L_p$ one has
$$
\Vert f -\varphi_n\Vert_{L_p}+\Vert g
-\varphi_n^{(r)}\Vert_{L_p}\to 0\quad\text{as}\quad n\to\infty,
$$
then (in the sense of distribution) $g=f^{(r)}$} (see~\cite[Ch. 4]{Nik}).

In the case $0<p<1$, this result does not valid in general. In particular, for $f_0(x)=x$ there exists a sequence of functions  $\varphi_n \in AC[0,1]$, $n\in\N$, such that
$\varphi_n \to f_0$ as $n\to\infty$ in $L_p[0,1]$, but $\Vert
\varphi_n'\Vert_{L_p[0,1]}\to 0$ as $n\to\infty$ (see~\cite{DiTi07}). This is an undesirable property of the spaces $L_p$, $0<p<1$.
However, as it is shown in Lemma~\ref{lem1w} below, under certain additional restrictions on $f$ and $\vp_n$, this feature can be fixed (see also~\cite{DiTi07}, in which the case of the derivatives of integer order was considered).

To prove the main result of this subsection (see Lemma~\ref{lem1w}),  we need the following two lemmas. As usual, the Fourier transform of a function $f\in L_1(\mathbb{R})$ is denoted by
$$
\widehat{f}(y)=\frac1{\sqrt{2\pi}}\int_{\mathbb{R}}f(x)e^{-iyx}dx.
$$

\begin{lemma}\label{lemBL} {\sc (See~\cite[4.1.1]{TB}).}
{\it Let $0<p\le 1$, a function $\phi\in C(\R)$ have a compact support, and  $\widehat{\phi}\in L_p(\R)$. Then
$$
\sup_{h>0} h^{1-\frac1p}\Vert
\Phi_h\Vert_{L_p(\mathbb{T})}=\sqrt{2\pi}\Vert\widehat{\phi}\Vert_{L_p(\mathbb{R})},
$$
where
$$
\Phi_h(x)=\sum_{k=-\infty}^\infty \phi\left(h k\right)e^{ikx}.
$$}
\end{lemma}

In the case $p=1$, the next lemma can be found in~\cite{K14}; the general case see in~\cite{K12}.

\begin{lemma}\label{lemFurp} {\it
    Let $0<p\le 1$, $1<q<\infty$, $1<r<\infty$, $s>1/p-1+1/r$,
    $s\in\N$, let a function $f$ be such that $f\in C(\R)\cap L_1(\R)$, $\lim_{|x|\to \infty}f(x)=0$, and $\widehat{f}\in L_1(\R)$. Suppose also that $f\in L_q(\R)$, $f^{(s)}\in L_r(\R)$, and
    $$
\frac{1-\t}q+\frac{\t}r>\frac12,\quad
\t=\frac1s\left(\frac1p-\frac12\right).
    $$
    Then
    $$
\Vert \widehat{f}\Vert_{L_p(\R)}\le C\Vert
f\Vert_{L_q(\R)}^{1-\t}\Vert f^{(s)}\Vert_{L_r(\R)}^{\t},
    $$
where $C$ is a constant independent of $f$.}
\end{lemma}

Now, we are ready to formulate and prove a key result for obtaining Theorems~\ref{thsec3.A} and~\ref{thModFrD}.

\begin{lemma}\label{lem1w}
{\it     Let $f\in L_p(\T)$, $0<p<1$,
$\a\in\N\cup(1/p-1,\infty)$, and
    $T_n\in \mathcal{T}_n$, $n\in \N$, be such that
    \begin{equation*}
        \Vert f-T_n\Vert_p=o\left(\frac1{n^{\a}}\right)\quad\text{and}\quad
        \Vert g-T_n^{(\a)}\Vert_p=o(1)\quad\text{as}\quad n\to\infty.
    \end{equation*}
Then $f^{(\a)}=g$, i.e. $g$ satisfies
(\ref{eqProizvLp}).}
\end{lemma}

\begin{proof}
For any sufficiently small $\e>0$, we choose $n_0=n_0(\e)$
such that for any $n\ge n_0$ one has
    \begin{equation}\label{eqlemditi2}
        \Vert f-T_n\Vert_p\le\frac\e{n^{\a}}\quad\text{and}\quad
        \Vert g-T_n^{(\a)}\Vert_p\le \e.
    \end{equation}

Let $h$ be such that ${\e^\l}{n^{-1}}\le
h\le{2\e^\l}{n^{-1}}$, where $0<\l<\a^{-1}$. We have
\begin{equation}\label{eq.lem.th2.3}
\begin{split}
\bigg\Vert \frac{\D_h^\a f}{h^\a}-g\bigg\Vert_{p}^p&\leq \bigg\Vert \frac{\D_h^\a (f-T_n)}{h^\a}\bigg\Vert_{p}^p\\
&+
\bigg\Vert \frac{\D_h^\a T_n}{h^\a}-T_n^{(\a)}\bigg\Vert_{p}^p+\Vert g-T_n^{(\a)}\Vert_{p}^p\\
&=J_1+J_2+J_3\,.
\end{split}
\end{equation}
Using~\eqref{eqM1} and (\ref{eqlemditi2}), we get
\begin{equation}\label{eqlemditi3}
\begin{split}
      J_1=\bigg\Vert \frac{\D_h^\a
    (f-T_n)}{h^\a}\bigg\Vert_p^p\le Ch^{-\a p}\Vert
    f-T_n\Vert_p^p \le C\e^{(1-\l\a)p},
\end{split}
\end{equation}
\begin{equation}\label{eq.lem.th2.3.5}
  J_3\leq\e^p.
\end{equation}
Let us consider $J_2$.
Set
\begin{equation}\label{ravenstvo}
    T_{n,h,\a}(t)=\frac{\D_h^\a
    T_n(t)}{h^\a}-T_n^{(\a)}(t).
\end{equation}
It is easy to see (here and throughout we use the principal branch of the logarithm) that
$$
T_{n,h,\a}(t)=\sum_{k=-n}^n (ik)^\a
\left(\left(\frac{1-e^{-ikh}}{ikh}\right)^\a-1\right) c_k e^{ikt},
$$
where $\{c_k\}_{k=-n}^n$ are the coefficients of $T_n$.
We also have the following equality
$$
T_{n,h,\a}(t)=(K_{h,\a}* T_n^{(\a)})(t),
$$
where
$$
K_{h,\a}(t)=\sum_{k\in \Z} \eta_{\a,\e}(hk) e^{ikt},
\quad
\eta_{\a,\e}(x)=\left(\left(\frac{1-e^{-ix}}{ix}\right)^{\a}-1\right)v\left(\frac
x{2\e^\l}\right),
$$
and the function $v$ is such that $|v(x)|\le 1$, $v\in C^\infty (\mathbb{R})$, $v(x)=1$ for
$|x|\le 1$ and $v(x)=0$ for $|x|\ge 2$.

Note that $K_{h,\a}(x)T_n^{(\a)}(t-x)$ is a trigonometric polynomial of order at most $4n$ in variable~$x$. Thus, using Lemma~\ref{lemNikpq}, we obtain
\begin{equation*}
  \begin{split}
     |T_{n,h,\a}(t)|^p&\le \bigg(\frac1{2\pi}\int_\T |K_{h,\a}(x) T_n^{(\a)}(t-x)|dx\bigg)^p\\
     &\le Cn^{1-p}\int_\T |K_{h,\a}(x) T_n^{(\a)}(t-x)|^p dx.
   \end{split}
\end{equation*}
Integrating the above inequality by $t$ and applying Fubini's theorem, we get
\begin{equation}\label{eq777}
\Vert T_{n,h,\a}\Vert_p\le C
    n^{\frac1p-1}\Vert K_{h,\a}\Vert_p \Vert T_n^{(\a)}\Vert_p.
\end{equation}
Now, let us consider the function $\eta_{\a,\e}$. Noting that for sufficiently small $x$
$$
\frac12\le \Big|\frac{e^{-ix}-1}{ix}\Big|\le 1,
$$
we derive for any $s=0,1,\dots$ the following estimates
\begin{equation*}
  \begin{split}
     |\eta_{\a,\e}^{(s)}(x)|&=\bigg|\sum_{\nu=0}^s \binom{s}{\nu} \bigg( \Big( \frac{e^{-ix}-1}{ix}\Big)^\alpha-1\bigg)^{(\nu)} \(v\Big(\frac{x}{2\e^\l}\Big)\)^{(s-\nu)} \bigg|\\
     &\le \frac1{\e^{s\lambda}} \sum_{\nu=0}^s c_{\nu,s,\a} \Big|v^{(s-\nu)} \Big(\frac{x}{2\e^\l}\Big)\Big|.
   \end{split}
\end{equation*}
Thus, it is easy to see that for any  $1<q<\infty$, $1<r<\infty$, and $s\in \N$, we have
$$
\Vert \eta_{\a,\e}\Vert_{L_q(\R)}\le C \e^{\frac\l q}\quad\text{and}\quad
\Vert \eta_{\a,\e}^{(s)}\Vert_{L_r(\R)}\le C \e^{\frac\l r-\l s}.
$$
Thus, by Lemmas~\ref{lemBL} and~\ref{lemFurp}, we derive
\begin{equation}\label{eqlemditi5}
    \begin{split}
 n^{\frac1p-1}\Vert K_{h,\a}\Vert_p&\le C \e^{\l(\frac1p-1)}h^{1-\frac1p}\Vert
K_{h,\a}\Vert_p\le C \e^{\l(\frac1p-1)} \Vert
\widehat{\eta_{\a,\e}}\Vert_{L_p(\R)}\\
&\le C\e^{\l((1-\t)\frac1q+\t\frac1r-\frac12)}=C\e^{\g}.
    \end{split}
\end{equation}
It is obvious that we can choose $q$ and $r$ such that $\g=\l((1-\t)/q+\t/r-1/2)>0$.
Then, using~(\ref{eq777}) and (\ref{eqlemditi5}), we get  $\Vert T_{n,h,\a}\Vert_p\le C\e^{\g}\Vert T_n^{(\a)}\Vert_p$. From this inequality,
tacking into account~\eqref{eqlemditi2} and~\eqref{ravenstvo}, we obtain
\begin{equation}\label{eqlemditi777}
J_2\le
C\e^{\g p}(\e^p+\Vert g\Vert_p^p).
\end{equation}

Finally, combining inequalities~\eqref{eq.lem.th2.3}--\eqref{eq.lem.th2.3.5} and
(\ref{eqlemditi777}), we derive
\begin{equation*}
\bigg\Vert \frac{\D_h^\a f}{h^\a}-g\bigg\Vert_p\le
C(\e^{1-\l\a}+\e^{\g}\Vert g\Vert_p+\e).
\end{equation*}
The last inequality implies that $f^{(\a)}=g$ in the sense
(\ref{eqProizvLp}).
\end{proof}

The next proposition shows that conditions of Lemma~\ref{lem1w} are sharp.
\begin{proposition}\label{pr1w}
{\it Let $0<p<1$ and $\a \in \N \cup
(1/p-1,\infty)$. Then there exists $f_\a \in L_1(\T)$ and a sequence of polynomials  $T_{n,\a}\in \mathcal{T}_n$, $n\in \N$, such that
$f_\a^{(\a)}(x)\equiv {\rm const}\neq 0$ a.e. on $[0,\pi)$ and
$$
\Vert f_\a-T_{n,\a}\Vert_p=\mathcal{O}\(\frac1{n^\a}\),\quad\text{but}\quad
\Vert T_{n,\a}^{(\a)}\Vert_p\to 0\quad \text{as}\quad n\to\infty,
$$
where $\asymp$ is a two-sided inequality with positive constants independent of $n$.}
\end{proposition}

\begin{proof}
We will use some ideas from~\cite{KP}.
Let $r\in \N$. Set
$$
f_r(x)=\left\{
       \begin{array}{ll}
        \displaystyle x^r, & \hbox{$x\in [0,\pi)$,} \\
        \displaystyle (2\pi-x)^r, & \hbox{$x\in [\pi,2\pi]$,}
       \end{array}
     \right.
$$
and
\begin{equation*}
g_{n,r}(x)=\left\{%
         \begin{array}{ll}
           \displaystyle\frac kn x^{r-1}, & \hbox{$\displaystyle\frac kn\le x <\frac{k+1}{n}-\frac1{n^{r+1}}$,} \\
           \phantom{.}\\
           \displaystyle\frac kn x^{r-1}+x^{r-1}\left(x-\frac{k+1}{n}+\frac1{n^{r+1}}\right)n, & \hbox{$\displaystyle\frac{k+1}{n}-\frac1{n^{r+1}}\le
x<\frac{k+1}{n}$,}\\
         \end{array}%
       \right.
\end{equation*}
for $k=0,1,\dots,n-1$, $g_{n,r}(x)=1-g_{n,r}(x-1)$ for $1<x\le 2$, and
$$
\varphi_{n,r}(x)=\pi g_{n,r}\(\frac{x}{\pi}\)\quad \text{for}\quad x\in
[0,2\pi).
$$

We need the following inequalities
\begin{equation}\label{eqKop+}
    \omega_r(\varphi_{n,r},n^{-1})_q\le Cn^{-r}\Vert \varphi_{n,r}^{(r)}\Vert_q\le
Cn^{-r-\frac1q+1},\quad 0<q<\infty .
\end{equation}
The first inequality can be found in~\cite{Kop}, the second one can be verified by simple calculation. It is also easy to see that
\begin{equation}\label{zzzzzzzz}
  \Vert f_r -\varphi_{n,r}\Vert_p=\mathcal{O}\left({n^{-r}}\right).
\end{equation}

Let $T_{n,r}\in \mathcal{T}_n$ be a polynomial of the best approximation of
$\varphi_{n,r}$ in $L_p$.
Using~\eqref{eqJ}, \eqref{zzzzzzzz}, and (\ref{eqKop+}), we obtain
\begin{equation}\label{eqvL1}
\begin{split}
   \Vert f_r -T_{n,r}\Vert_p
   &\le C(\Vert
f_r-\varphi_{n,r}\Vert_p+\Vert \varphi_{n,r}-T_{n,r}\Vert_p)\\
&\le C(n^{-r}+\omega_r(\varphi_{n,r},n^{-1})_p)\le C n^{-r}.
\end{split}
\end{equation}
At the same time, by Lemma~\ref{lem1} and (\ref{eqKop+}), one has
\begin{equation}\label{eqvLpDer}
    \Vert T_{n,r}^{(r)}\Vert_p\le C n^r \omega_r(\varphi_{n,r},n^{-1})_p\le C
n^{1-\frac1p}.
\end{equation}
Thus, we have proved the proposition in the case $\alpha=r\in \N$.

Now let $\alpha\not\in \N$. Choose $r\in \N$ such that
$r>\alpha$ and denote $f_\alpha=f_r^{(r-\alpha)}$ and
$T_{n,\a}=T_{n,r}^{(r-\a)}$.
Note that if $f\in L_p(\mathbb{T})$, $\gamma>\b>1/p-1$, and $T_{n}\in \mathcal{T}_n$, $n\in \N$,
are such that
\begin{equation}\label{-}
  \Vert f-T_{n}\Vert_p=\mathcal{O}(n^{-\gamma})\quad\text{as}\quad n\to
\infty,
\end{equation}
then $f$ has the derivative $f^{(\b)}$ in the sense of $L_p$ and
\begin{equation}\label{--}
  \Vert
f^{(\b)}-T_{n}^{(\b)}\Vert_p=\mathcal{O}(n^{-(\gamma-\b)})\quad\text{as}\quad
n\to\infty.
\end{equation}
This can be verified repeating the proof of Theorem~\ref{thsec3.A} presented below.
Thus, using
(\ref{eqvL1}) and taking into account~\eqref{-} and~\eqref{--},  we obtain
$$
\Vert f_\alpha -T_{n,\alpha}\Vert_p=\Vert f_r^{(r-\alpha)}
-T_{n,r}^{(r-\a)}\Vert_p\le Cn^{-\alpha}.
$$
At the same time, by (\ref{eqvLpDer}), we get
$$
\Vert T_{n,\alpha}^{(\alpha)}\Vert_p=\Vert
T_{n,r}^{(r)}\Vert_p\le Cn^{1-\frac1p}.
$$
The last two inequalities prove the proposition.
\end{proof}

\bigskip

\section{Proofs of the main results}

\begin{proof}[Proof of Theorem~\ref{thsec3.1}]
It is clear that we can assume that
\begin{equation}\label{eqthsec3.1.1-}
  \sum\limits_{\nu=1}^\infty \nu^{-p}E_\nu(f^{(\a)})_p^p<\infty\,.
\end{equation}

Let $U_n\in\mathcal{T}_n$ and $T_n \in\mathcal{T}_n$, $n\in\N$, be such that
$$
\|f^{(\a)}-U_n\|_p=E_n(f^{(\a)})_p
$$
and
$$
T_n^{(\a)}(x)=U_n(x)-\frac{1}{2\pi}\int_0^{2\pi} U_n(x){d}x.
$$

Choosing $m\in\N$ such that $2^{m-2}\leq n<2^{m-1}$, we derive
\begin{equation}\label{eqthsec3.1.3}
  E_n(f)_p^p\leq E_n(T_{2^m})_p^p+E_n(f-T_{2^m})_p^p\,.
\end{equation}
Let us estimate $E_n(T_{2^m})_p$. Set
$$
\tau_u(x)=\tau_{u,2^m,n}(x)=\D_u^1(T_{2^m}(x)-T_n(x)),\quad u>0\,.
$$
Applying~\eqref{eqJ} with $\b=\a+r$, $r>1/p$, and Lemma~\ref{LemKfunc}, we obtain
\begin{equation}\label{eqthsec3.1.4}
\begin{split}
E_n(T_{2^m})_p&=E_n(T_{2^m}-T_n)_p\leq C\w_{\a+r}(T_{2^m}-T_n,\,n^{-1})_p\\
&=C \sup\limits_{0<h\leq n^{-1}}\|\D_h^{\a+r-1} \tau_h\|_p\leq C \sup\limits_{0<h\leq n^{-1}} \sup\limits_{u>0}\|\D_h^{\a+r-1} \tau_u\|_p\\&\leq C \sup\limits_{u>0}\w_{\a+r-1}(\tau_u,\,n^{-1})_p\leq C \sup\limits_{u>0}\mathcal{R}_{\a+r-1}(\tau_u,\,n^{-1})_p\,.
\end{split}
\end{equation}
Next, let $V_n\in \mathcal{T}_n$, $n\in \N$, be such that
\begin{equation}\label{eqKfunct+}
  \Vert \tau_u-V_n\Vert_p+n^{-\a}\Vert V_n^{(\a)}\Vert_p\leq 2 \mathcal{R}_\a(\tau_u,n^{-1})_p.
\end{equation}
Then, by the definition of the realization $\mathcal{R}_\a$, using Lemmas~\ref{lemBLf} and~\ref{LemKfunc+}, inequalities~\eqref{eqKfunct+} and~(\ref{eqM1}), and  taking into account that $\tau_u\in \mathcal{T}_{2^m}$ for any fixed $u>0$, we get
\begin{equation}\label{eqthsec3.1.5}
\begin{split}
\mathcal{R}_{\a+r-1}(\tau_u,\,n^{-1})_p&\leq \Vert \tau_u -V_n\Vert_p+n^{-(\a+r-1)}\Vert V_n^{(\a+r-1)}\Vert_p\\
&\leq \Vert \tau_u -V_n\Vert_p+C n^{-\a}\Vert V_n^{(\a)}\Vert_p\leq C \mathcal{R}_\a(\tau_u,n^{-1})_p\\
&\leq C \mathcal{R}_\a(\tau_u,2^{-m})_p
\leq C 2^{-m\a}\|\tau_u^{(\a)}\|_p\\
&=C 2^{-m\a}\|\D_u^1(T_{2^m}^{(\a)}-T_n^{(\a)})\|_p=C 2^{-m \a}\|\D_u^1(U_{2^m}-U_n)\|_p\\
&\leq C n^{-\a}\|U_{2^m}-U_n\|_p\leq C n^{-\a} E_n(f^{(\a)})_p\,.
\end{split}
\end{equation}
Combining (\ref{eqthsec3.1.4}) and (\ref{eqthsec3.1.5}), we derive
\begin{equation}\label{eqthsec3.1.6}
E_n(T_{2^m})_p\leq C n^{-\a} E_n(f^{(\a)})_p\,.
\end{equation}

Now, let us consider the second term in the right-hand side of (\ref{eqthsec3.1.3}).
For any $N>m$, we have
\begin{equation}
\label{eqthsec--3.1.8}
E_n(f-T_{2^m})_p^p\leq \sum\limits_{\mu=m}^{N-1} E_n(T_{2^{\mu+1}}-T_{2^\mu})_p^p+E_n(f-T_{2^N})_p^p\,.
\end{equation}
Using H\"older's inequality and~\eqref{+E}, we obtain
\begin{equation}\label{XXXX}
\begin{split}
E_n\(f-T_{2^N}\)_p&\leq C E_n\(f-T_{2^N}\)_1\leq C n^{-\a} E_n(f^{(\a)}-T_{2^N}^{(\a)})_1\\
&= C n^{-\a} E_n(f^{(\a)}-U_{2^N})_1\leq C n^{-\a} \|f^{(\a)}-U_{2^N}\|_1\,.
\end{split}
\end{equation}
Let us show that $\|f^{(\a)}-U_{2^N}\|_1\to 0$ as $N\to \infty$.
By Lemma~\ref{lemNikpq}, we have
\begin{equation*}
\begin{split}
\sum\limits_{\mu=1}^\infty\|U_{2^{\mu+1}}-U_{2^\mu}\|_1^p&\leq C\sum\limits_{\mu=1}^\infty 2^{(1-p)\mu}\|U_{2^{\mu+1}}-U_{2^\mu}\|_p^p\\
&\leq C\sum\limits_{\mu=1}^\infty 2^{(1-p)\mu} E_{2^\mu}(f^{(\a)})_p^p\leq C\sum\limits_{\nu=1}^\infty \nu^{-p} E_{\nu}(f^{(\a)})_p^p\,.
\end{split}
\end{equation*}
In view of~\eqref{eqthsec3.1.1-}, this implies that there exists $g\in L_1(\T)$ such that
$U_{2^\mu}\rightarrow g$ as $\mu\rightarrow\infty$ in $L_1(\T)$. By the definition of $U_n$, we know that $U_{2^\mu}\rightarrow f^{(\a)}$ as
$\mu\rightarrow\infty$ in $L_p(\T)$. Therefore, $g=f^{(\a)}$ a.e. on $\T$ and
\begin{equation}\label{zvezda}
  U_{2^\mu}\rightarrow f^{(\a)}\quad \text{as}\quad \mu\rightarrow\infty\quad \text{in}\quad L_1(\T).
\end{equation}
Thus, combining~\eqref{eqthsec--3.1.8} and~\eqref{XXXX} and taking into account~\eqref{zvezda},  we get
\begin{equation}\label{eqthsec3.1.7}
E_n(f-T_{2^m})_p^p\leq \sum\limits_{\mu=m}^\infty E_n(T_{2^{\mu+1}}-T_{2^\mu})_p^p\,.
\end{equation}

Now, applying the same arguments as in \eqref{eqthsec3.1.4} and (\ref{eqthsec3.1.5}) to the function
$\tau_u(x)=\tau_{u,2^{\mu+1},2^\mu}(x)=\D_u^1(T_{2^{\mu+1}}(x)-T_{2^\mu}(x))$, we derive
\begin{equation}\label{eqthsec3.1.10}
\begin{split}
E_n(T_{2^{\mu+1}}-T_{2^\mu})_p& \leq C\w_{\a+r}(T_{2^{\mu+1}}-T_{2^\mu},\,n^{-1})_p\leq C \sup\limits_{u>0}\mathcal{R}_\a(\tau_u,\,n^{-1})_p\\
&\leq C(2^{\mu+1} n^{-1})^{\a+\frac{1}{p}-1}\sup\limits_{u>0}\mathcal{R}_\a(\tau_u,\,2^{-\mu-1})_p\\
&\leq C n^{-\a-\frac{1}{p}+1}2^{\mu\(\frac{1}{p}-1\)}\sup\limits_{u>0}\|\tau_u^{(\a)}\|_p.\\
\end{split}
\end{equation}
Note that in the third inequality, we use Lemma~\ref{LemKfunc+} and  take into account that $n<2^{m-1}\le 2^{\mu+1}$. Next, by (\ref{eqM1})
\begin{equation}\label{+eqthsec3.1.10+}
\begin{split}
\|\tau_u^{(\a)}\|_p&=\Vert \D_u^1 (T_{2^{\mu+1}}^{(\a)}-T_{2^{\mu}}^{(\a)})\Vert_p=\Vert \D_u^1 (U_{2^{\mu+1}}-U_{2^{\mu}})\Vert_p\\
&\leq C\Vert U_{2^{\mu+1}}-U_{2^{\mu}}\Vert_p \leq  CE_{2^\mu}(f^{(\a)})_p\,.
\end{split}
\end{equation}
Combining (\ref{eqthsec3.1.7})--(\ref{+eqthsec3.1.10+}), we obtain
\begin{equation}\label{eqthsec3.1.11}
\begin{split}
E_n(f-T_{2^m})_p^p& \leq Cn^{-\a p-1+p}\sum\limits_{\mu=m}^\infty 2^{(1-p)\mu} E_{2^\mu}(f^{(\a)})_p^p\\
&\leq Cn^{-\a p-1+p}\sum\limits_{\nu=n+1}^\infty \nu^{-p} E_{\nu}(f^{(\a)})_p^p\,.
\end{split}
\end{equation}

Finally, combining (\ref{eqthsec3.1.3}), (\ref{eqthsec3.1.6}), and (\ref{eqthsec3.1.11}), we get (\ref{eqthsec3.1.2}).
\end{proof}

\begin{proof}[Proof of Theorem~\ref{thsec3.A}]
Let $N\in \N$ be such that $2^{N-1}\le n<2^N$. Assuming for a moment
that $f^{(\a)}$ exists, we get
\begin{equation}\label{eqth2S1}
  \Vert f^{(\a)}-T_n^{(\a)}\Vert_p^p\le \Vert f^{(\a)}-T_{2^N}^{(\a)}\Vert_p^p+\Vert T_{2^N}^{(\a)}-T_n^{(\a)}\Vert_p^p.
\end{equation}

By Lemma~\ref{lemBLf}, we obtain
\begin{equation}\label{eqth2S2}
\begin{split}
  \Vert T_{2^N}^{(\a)}-T_n^{(\a)}\Vert_p^p\le C2^{\a Np}  \Vert T_{2^N}-T_n\Vert_p^p
  \le Cn^{\a p} E_{n}(f)_p^p
\end{split}
\end{equation}
and
\begin{equation}\label{eqth2S3}
\begin{split}
\sum_{\nu=N}^\infty \Vert T_{2^{\nu+1}}^{(\a)}-T_{2^\nu}^{(\a)}\Vert_p^p  &\le
C \sum_{\nu=N}^\infty 2^{\a p\nu} \Vert
T_{2^{\nu+1}}-T_{2^\nu} \Vert_p^p\\
&\le C \sum_{\nu=N}^\infty 2^{\a p\nu}E_{2^\nu}(f)_p^p\,.
\end{split}
\end{equation}
Thus, by the completeness of $L_p(\T)$ and condition
(\ref{eqthsec3.A.1}), there exists a function $g\in L_p(\T)$ such that
\begin{equation}\label{eqprthModFrD.7}
\begin{split}
\Vert g-T_{2^N}^{(\a)}\Vert_p=\lim_{l\to\infty}\Vert
T_{2^l}^{(\a)}-T_{2^N}^{(\a)}\Vert_p\le C \(\sum_{\nu=N}^\infty
2^{\a p\nu}E_{2^\nu}(f)_p^p\)^\frac 1p.
\end{split}
\end{equation}
In~\eqref{eqprthModFrD.7}, we use the equality $
T_{2^l}-T_{2^N}=\sum_{\nu=N}^{l-1}(T_{2^{\nu+1}}-T_{2^\nu})$ and~\eqref{eqth2S3}. It is also easy to see that
\begin{equation}\label{eqprthModFrD.8}
\begin{split}
\Vert f-T_{2^N}\Vert_p\leq C 2^{-N \a}\(2^{N
\a}E_{2^N}(f)_p\)=o(2^{-N \a})\quad \mathrm{as}\quad N\to \infty.
\end{split}
\end{equation}
Therefore, by Lemma~\ref{lem1w},  (\ref{eqprthModFrD.8}), and (\ref{eqprthModFrD.7}),  we obtain that $g=f^{(\a)}$.

Finally, combining (\ref{eqth2S1}), (\ref{eqth2S2}), and (\ref{eqprthModFrD.7}), we get (\ref{eqthsec3.A.2}).
\end{proof}

\begin{proof}[Proof of Theorem~\ref{propanal1}]
The proof is similar to the proof of Theorem~\ref{thsec3.A} using Lemma~\ref{lemBLf} for all $\a>0$.
\end{proof}

\begin{proof}[Proof of Theorem~\ref{thsec3.2}]
We prove the theorem only in the case $\a\in \N\cup (1/p-1,\infty)$. The other cases for $\a$ can be obtained repeating the arguments presented below.

Using (\ref{eqthsec3.1.2}), we obtain
\begin{equation}\label{eqthsec3.2.3}
\begin{split}
&\sum\limits_{\nu=n+1}^\infty\nu^{\a p-1}E_\nu(f)_p^p\\
&\leq C\sum\limits_{\nu=n+1}^\infty \(\nu^{-1}E_\nu(f^{(\a)})_p^p+\nu^{p-2}\sum\limits_{\mu=\nu+1}^\infty \mu^{-p}E_\mu(f^{(\a)})_p^p\)\\
&\leq C\sum\limits_{\nu=n+1}^\infty \nu^{p-1}\nu^{-p}E_\nu(f^{(\a)})_p^p+C\(\sum\limits_{\nu=n+1}^\infty\nu^{p-2}\)\sum\limits_{\mu=n+1}^\infty \mu^{-p}E_\mu(f^{(\a)})_p^p\\
&\leq Cn^{p-1}\sum\limits_{\nu=n+1}^\infty \nu^{-p}E_\nu(f^{(\a)})_p^p\,.
\end{split}
\end{equation}
Therefore, combining (\ref{eqthsec3.A.2+A}), (\ref{eqthsec3.1.2}), and (\ref{eqthsec3.2.3}), we get (\ref{eqthsec3.2.2}).
\end{proof}

\begin{proof}[Proof of Theorem~\ref{thDif}]

Let $n\in \N$ be such that $1/(n+1)<\d\le 1/n$ and let $T_n\in\mathcal{T}_n$ be polynomials of the best approximation of $f$ in $L_p(\T)$.
By~(\ref{eqM0}), we get
\begin{equation}\label{eqth1T.1}
\begin{split}
  \w_{\a+\b}(f,\d)_p^p&\le \w_{\a+\b}(f,1/n)_p^p\\
  &\le \w_{\a+\b}(f-T_n,1/n)_p^p+\w_{\a+\b}(T_n,1/n)_p^p=M_1+M_2.
\end{split}
\end{equation}
Using Lemma~\ref{lem1}, (\ref{eqM0}), and (\ref{eqM1}), we obtain
\begin{equation}\label{eqDif5T}
\begin{split}
    M_2&\le  Cn^{-\a p}\w_\b (T_n^{(\a)},1/n)_p^p\\
&\le Cn^{-\a p}\(\Vert f^{(\a)}-T_n^{(\a)}\Vert_p^p+\w_\b(f^{(\a)},1/n)_p^p\).
\end{split}
\end{equation}
Next, by Theorem~\ref{thsec3.2} and the Jackson--type inequality~\eqref{eqJ}, we have
\begin{equation}\label{eqDif6T}
\begin{split}
\Vert f^{(\a)}-T_n^{(\a)}\Vert_p^p&\le C \(\w_r(f^{(\a)},1/n)_p^p+n^{p-1}\sum\limits_{\nu=n+1}^\infty \nu^{-p}\w_r(f^{(\a)},1/\nu)_p^p\)\\
&\le Cn^{p-1}\int_0^{1/n} \frac{\w_r(f^{(\a)},t)_p^p}{t^{2-p}}{d}t.
\end{split}
\end{equation}
At the same time, by~(\ref{eqM1}),  Theorem~\ref{thsec3.1}, and~\eqref{eqJ}, we derive
\begin{equation}\label{eqDif7T}
\begin{split}
M_1&\le C\Vert f-T_n\Vert_p^p\\
&\le Cn^{-\a p}\(\w_r(f^{(\a)},1/n)_p^p+n^{p-1}\sum\limits_{\nu=n+1}^\infty\nu^{-p}\w_r(f^{(\a)},1/\nu)_p^p\) \\
&\le Cn^{p-1-\a p}\int_0^{1/n} \frac{\w_r(f^{(\a)},t)_p^p}{t^{2-p}}{d}t.
\end{split}
\end{equation}
Thus, combining (\ref{eqth1T.1})--(\ref{eqDif7T}) and taking into
account (\ref{eqM2}) and $1/(n+1)<\d\le 1/n$, we get~(\ref{eqDif1}).
\end{proof}

\begin{proof}[Proof of Theorem~\ref{thModFrD}]
The proof is similar to the proof of Theorem~\ref{thDif} combining Theorem~\ref{thsec3.A}, Lemma~\ref{lem1}, and the Jackson inequality~\eqref{eqJ}.
\end{proof}

\begin{proof}[Proof of Theorem~\ref{thModFrD10}]
The proof is similar to the proof of Theorems~\ref{thDif} and~\ref{thModFrD}. We only note that we use Theorem~\ref{propanal1} instead of Theorem~\ref{thsec3.A}.
\end{proof}

\begin{proof}[Proof of Theorem~\ref{thsec3.1++++}] The proof easily follows from inequality~\eqref{eqJ} and Theorem~\ref{thsec3.1}.
\end{proof}

\begin{proof}[Proof of Theorem~\ref{thsec3.A++++}]
In view of~\eqref{eqthsec3.A.1++++} and Theorem~\ref{thsec3.A}, the function  $f$ has the derivative $f^{(\a)}$ in the sense of $L_p$.
Using inequalities~\eqref{eqthsec3.A.2} and~\eqref{eqB}, we obtain
\begin{equation}\label{mo1}
  \begin{split}
    \w_\b(f^{(\a)},n^{-1})_p^p&\le Cn^{-\b p}\sum_{\nu=0}^n (\nu+1)^{\b p-1}E_\nu(f^{(\a)})_p^p\\
   &\le Cn^{-\b p}\sum_{\nu=0}^n (\nu+1)^{\b p-1}\bigg((\nu+1)^{\a p}E_\nu(f)_p^p+\sum_{\mu=\nu+1}^\infty \mu^{\a p-1}E_\mu(f)_p^p\bigg)\\
   &=Cn^{-\b p}\Bigg(\sum_{\nu=0}^n (\nu+1)^{\a p+\b p-1}E_\nu(f)_p^p\\
&\quad\quad\quad\quad+\sum_{\nu=0}^n (\nu+1)^{\b p-1}\bigg(\sum_{\mu=\nu}^n+ \sum_{\mu={n+1}}^\infty\bigg) (\mu+1)^{\a p-1}E_{\mu+1}(f)_p^p\Bigg).
  \end{split}
\end{equation}
Further, we have
\begin{equation}\label{mo2}
  \begin{split}
   \sum_{\nu=0}^n (\nu+1)^{\b p-1}&\sum_{\mu=\nu}^n (\mu+1)^{\a p-1}E_{\mu+1}(f)_p^p\\
   &=\sum_{\mu=0}^n (\mu+1)^{\a p-1}E_{\mu+1}(f)_p^p \sum_{\nu=0}^\mu (\nu+1)^{\b p-1} \\
   &\le C\sum_{\mu=0}^n (\mu+1)^{\a p + \b p-1}E_{\mu}(f)_p^p.
  \end{split}
\end{equation}
At the same time, we derive
\begin{equation}\label{mo3}
  \begin{split}
   n^{-\b p}\sum_{\nu=0}^n (\nu+1)^{\b p-1}&\sum_{\mu=n+1}^\infty (\mu+1)^{\a p-1}E_{\mu+1}(f)_p^p\le C\sum_{\mu=n+1}^\infty\mu^{\a p-1}E_\mu(f)_p^p.
  \end{split}
\end{equation}

Finally, combining~\eqref{mo1}--\eqref{mo3}, we get~\eqref{eqthsec3.1.2++++}.
\end{proof}

\begin{proof}[Proof of Theorem~\ref{-thsec3.A++++}]
The proof is similar to the proof of Theorem~\ref{thsec3.A++++} by using inequality~\eqref{eqthsec3.A.2+A} instead of~\eqref{eqthsec3.A.2}.
\end{proof}

\begin{proof}[Proof of Theorem~\ref{th2}]

First, we show that $(i)$ implies $(ii)$.
By Ulynov's type inequality (see, e.g.,~\cite{diti}) and inequality~\eqref{eqM1}, for any $r>\b+1/p-1$, $r\in \N$, we have
\begin{equation*}
  \begin{split}
    \Vert f\Vert_1^p&\le C\(\int_0^1\(\frac{\w_r(f,t)_p}{t^{1/p-1}}\)^p\frac{dt}{t}+\Vert f\Vert_p^p\)\\
    &\le C\(\int_0^1\(\frac{\w_\b(f,t)_p}{t^{1/p-1}}\)^p\frac{dt}{t}+\Vert f\Vert_p^p\)\le C\(\int_0^1 t^{\b p-1}dt+\Vert f\Vert_p^p \)<\infty,
  \end{split}
\end{equation*}
that is $f\in L_1(\T)$.
Next, using Theorem~\ref{thModFrD} in the case $\b\ge 1$ and Theorem~\ref{thDif} in the case $\b<1$, we get
\begin{equation*}
    \w_1(f^{(\b-1)},\d)_p\le C\(\int_0^\d
    \frac{\w_\b(f,t)_p^p}{t^{(\b-1)p+1}}dt\)^\frac1p=\mathcal{O}(\d^\frac1p),\quad \b\ge 1,
\end{equation*}
and
\begin{equation*}
  \w_1(I_{1-\b}f,\d)_p\le C\d^{\frac1p-\b}\(\int_0^\d\frac{\w_\b(f,t)_p^p}{t^{2-p}}dt\)^\frac1p=\mathcal{O}(\d^\frac1p),\quad \b<1,
\end{equation*}
where, for the clarity, we use the notation
$$
I_{\a}f=f^{(-\a)}
$$
to denote the fractional integral of order $\a>0$. It only remains to apply Proposition~\ref{propKrot}.

Now, let us prove that $(ii)$ implies $(i)$.
Let
$$
f^{(\b-1)}(x)=d_0+\sum_{x_k<x} d_k=d_0+\sum_{k=1}^\infty d_k
h_{x_k}(x),
$$
where
$$
h_\eta(x)=\left\{
            \begin{array}{ll}
              1, & \hbox{$x>\eta$,} \\
              0, & \hbox{$x\le \eta$.}
            \end{array}
          \right.
$$
Then, we have
$f(x)=d_0'+\sum_{k=1}^\infty d_k I_{\b-1}h_{x_k}(x)$. Using~\eqref{eqM0}, we get
\begin{equation}\label{eqprth2.1}
    \w_\b (f,\d)_p^p\le \sum_{k=1}^\infty |d_k|^p \w_\b (I_{\b-1}
    h_{x_k},\d)_p^p.
\end{equation}

Next, for any $r\in \N$ and $\eta\in \R$, we have
\begin{equation}\label{eqprth2.2}
    \w_r (I_{r-1}
    h_{\eta},\d)_p\le C(p,r)\d^{r+\frac1p-1}
\end{equation}
(see, e.g.,~\cite{Ra} or~\cite[p.~359]{DeLo}).
Choose $\a>1/p-1$ such that  $\b+\a=r\in \N$. Then, applying Theorem~\ref{thModFrD} and (\ref{eqprth2.2}), we obtain
\begin{equation}\label{eqprth2.3}
\begin{split}
    \w_\b (I_{\b-1} h_\eta,\d)_p&\le C\(\int_0^\d \frac{\w_{\a+\b}(I_\a I_{\b-1} h_\eta,t)_p^p}{t^{\a
    p+1}}dt\)^\frac1p\\
    &=C\(\int_0^\d \frac{\w_{r}(I_{r-1} h_\eta,t)_p^p}{t^{\a
    p+1}}dt\)^\frac1p\le C(p,\a,\b)\d^{\b+\frac1p-1}.
\end{split}
\end{equation}

Finally, combining~(\ref{eqprth2.1})
and (\ref{eqprth2.3}), we prove the theorem.
\end{proof}


\begin{thebibliography}{16}
{\small

\bibitem{BL}  E.~Belinsky, E.~Liflyand,  Approximation properties in $L_p$, $0<p<1$,
 Funct. Approx. Comment. Math.~\textbf{22} (1993), 189--199.


\bibitem{br2} Yu. A. Brudnyi, Criteria for the existence of derivatives in $L_p$,
Math. USSR-Sb. \textbf{2} (1) (1967), 35--55.

\bibitem{br3} Yu. A. Brudnyi, On the saturation class of a spline approximation with uniform nodes in $L_p$, $0<p<\infty$,
Issled. Teor. Funkts. Mnogikh Veshchestv. Perem. (1984), 34--40 (in Russian).

\bibitem{BDGS77}  P. L. Butzer, H. Dyckhoff, E. G\"{o}rlich, R. L. Stens, Best trigonometric approximation, fractional order derivatives
and Lipschitz classes, Can. J. Math. \textbf{ 29} (1977), 781--793.

\bibitem{BW} {P. L. Butzer, U. Westphal},
An access to fractional differentiation via fractional difference
quotients, Lecture Notes in Math., vol. 457.  Berlin:
Springer, 1975, pp. 116--145.

\bibitem{Ci} P. Civin, Inequalities for trigonometric integrals, Duke Math. J. \textbf{8} (1941), 656--665.

\bibitem{Di} Z. Ditzian, A Note on simultaneous polynomial approximation in $L_p[-1,1]$, $0 < p < 1$, J.~Approx. Theory \textbf{82} (2) (1995), 317--319.

\bibitem{CzFr} J. Czipszer,  G. Freud, Sur l’approximation d’une fonction p\'eriodique et de ses d\'eriv\'ees successives par un polynome trigonom\'etrique et par ses d\'eriv\'ees successives, Acta Math. \textbf{99} (1958), 33--51.

\bibitem{DeLo}    {R. A. DeVore, G. G. Lorentz},  Constructive Approximation, Springer-Verlag, New York, 1993.


\bibitem{DHI} Z.~Ditzian, V.~Hristov, K.~Ivanov, {Moduli of smoothness and
$K$-functional in $L_p$, $0<p<1$}, Constr. Approx. \textbf{11}
(1995), 67--83.


\bibitem{diti}
Z. Ditzian, S. Tikhonov,
Ul'yanov and Nikol'skii-type inequalities, J. Approx. Theory  {\bf 133}  (1)  (2005), 100--133.


\bibitem{DiTi07} Z. Ditzian, S. Tikhonov, Moduli of smoothness of
functions and their derivatives,  Studia Math. {\bf 180} (2) (2007),
143--160.




\bibitem{I}  V. I. Ivanov, Direct and inverse theorems of approximation theory in the metrics $L_p$ for $0<p<1$,
Math. Notes \textbf{18} (5) (1975),  972--982.

\bibitem{johnen}
H. Johnen, K. Scherer, On the equivalence of the $K$-functional and moduli of
continuity and some applications, Constructive theory of functions of several variables
(Proc. Conf., Math. Res. Inst., Oberwolfach 1976), Lecture Notes in Math., Vol. 571,
Springer-Verlag, Berlin–Heidelberg, 1977, pp. 119--140.

\bibitem{hl} {G. H. Hardy, J. Littlewood},  Some properties of
                   fractional integrals, I.-Math. Z.  \textbf{27} (1928),  565--606.


\bibitem{kolmod}  Yu. S. Kolomoitsev, Description of a class of
functions with the condition $\omega_r(f,h)_p\le Mh^{r-1+1/p}$ for
$0<p<1$,  Vestn. Dnepr. Univ., Ser. Mat.  \textbf{8} (2003), 31--43 (in Russian).


\bibitem{K07}
Yu. Kolomoitsev,
The inequality of Nikol’skii-Stechkin-Boas type with fractional derivatives in $L_p$, $0<p<1$,
Tr. Inst. Prikl. Mat. Mekh. \textbf{15} (2007), 115--119 (in Russian).

\bibitem{K11}
Yu. Kolomoitsev,
On moduli of smoothness and $K$-functionals of fractional order in the Hardy spaces,
J. Math. Sci. {\bf 181} (1) (2012), 78--97; translation from Ukr. Mat. Visn. {\bf 8} (3) (2011), 421--446.



\bibitem{K12}  Yu. Kolomoitsev,
On a class of functions representable as a Fourier integral,
Tr. Inst. Prikl. Mat. Mekh. {\bf 25} (2012), 125--132 (in Russian).

\bibitem{K14} Yu. Kolomoitsev,
Multiplicative sufficient conditions for Fourier multipliers,
Izv. Math. {\bf 78} (2) (2014), 354--374; translation from Izv. Ross. Akad. Nauk, Ser. Mat. {\bf 78}  (2) (2014), 145--166.


\bibitem{KoArh} Yu. Kolomoitsev, Best approximations and moduli of smoothness of functions and their derivatives in $L_p$, $0<p<1$,  J.~Approx. Theory \textbf{232} (2018), 12--42.

\bibitem{KP}  Yu. Kolomoitsev,  J. Prestin, Sharp estimates of approximation of periodic functions in H\"older spaces, J.~Approx. Theory \textbf{200} (2015), 68--91.

\bibitem{Kop95} K. A. Kopotun, On $K$-monotone polynomial and spline approximation in $L_p$, $0<p<\infty$  (quasi)norm. Approximation Theory VIII, World Scientific Publishing Co., C. Chui and L. Schumaker (eds.), 1995, pp. 295--302.

\bibitem{Kop} {K. A. Kopotun}, On equivalence of moduli of smoothness of
splines in $L_p$, $0<p<1$, J. Approx. Theory \textbf{143} (1) (2006) 36--43.

\bibitem{Kr} V. G. Krotov, On differentiability of functions in $L_p$, $0<p<1$,
Sb. Math. USSR \textbf{25 }(1983), 101--119.

\bibitem{Nik}  S. M. Nikol'skii, Approximation of functions of several variables and imbedding theorems, Springer-Verlag, New York-Heidelberg, 1975.

\bibitem{PST2016} M. K. Potapov, B. V. Simonov, S. Yu. Tikhonov,
Fractional  moduli of smoothness,  Max Press, Moscow, 2016.

\bibitem{peetre}
J. Peetre, A remark on Sobolev spaces. The case $0<p<1$,  J. Approx. Theory {\bf 13} (1975), 218--228.

\bibitem{PePo}  P. P. Petrushev,  V. A. Popov, Rational Approximation of Real Functions, Cambridge University Press, Cambridge, UK, 1987.

\bibitem{Ra} T. V. Radoslavova,  Decrease orders of the $L^p$-moduli of continuity ($0<p\le \infty$), Anal. Math. \textbf{5} (3) (1979), 219--234.


\bibitem{run}  K. Runovski, Approximation of families of linear polynomial
operators, Disser. of Doctor of Science, Moscow State University,  2010.



\bibitem{RS3} K.  Runovski, H.-J. Schmeisser,
General Moduli of Smoothness and Approximation by Families of Linear Polynomial Operators,
New Perspectives on Approximation and Sampling Theory,
Applied and Numerical Harmonic Analysis, 2014, pp. 269--298.

\bibitem{Ru17} K. V. Runovskii,  Approximation by trigonometric polynomials, $K$-functionals and generalized moduli of smoothness, Sb. Math. \textbf{208} no. 1-2, (2017),  237--254.


\bibitem{SKM}   S. G. Samko, A. A. Kilbas, O. I. Marichev, Fractional Integrals and Derivatives, in: Theory and Applications, Gordon
and Breach, Yverdon, 1993.


\bibitem{ST} B. V. Simonov, S. Yu.  Tikhonov,
Embedding theorems in constructive approximation,
Sb. Math. \textbf{199} (9) (2008), 1367--1407; translation from Mat. Sb. \textbf{199} (9) (2008), 107--148.

\bibitem{ST10} B. Simonov, S. Tikhonov, Sharp Ul'yanov-type inequalities
using fractional smoothness,  J. Approx. Theory {\bf 162}
 (9) (2010), 1654--1684.



\bibitem{SO} E. A. Storozhenko, P. Oswald, Jackson's theorem in the spaces $L_p(\R^k)$, $0<p<1$,
Sib. Math. J.  {\bf 19} (4) (1978), 630--656.

\bibitem{SKO75} E. A. Storozhenko, V. G. Krotov,  P. Oswald, Direct and inverse theorems of Jackson type in the space $L_p$, $0<p<1$,
Mat. Sb. {\bf 98} (3) (1975), 395--415.

\bibitem{tabeR}
R. Taberski, Differences, moduli and derivatives of fractional orders, Commentat. Math. {\bf 19} (1976-77), 389--400.

\bibitem{Tab_79}  R. Taberski,  Approximation in the Fr\'echet spaces  $L_p$ ($0<p\le 1$),
 Funct. Approx. Comment. Math.~\textbf{7} (1979), 105--121.

\bibitem{Tab_book} R. Taberski, Aproksymacja funkcji wielomianami trygonometrycznymi, Wydawnictwo Naukowe UAM, Poznan, 1979 (in Polish).

\bibitem{Tab_84}  R. Taberski, Trigonometric approximation in the norms and seminorms, Studia Math. \textbf{80} (1984), 197--217.

\bibitem{Ti05}   S. Tikhonov, On modulus of smoothness of fractional order, Real Anal. Exchange \textbf{30}  (2004/2005), 507--518.

\bibitem{timan} A. F. Timan, Theory of Approximation of Functions of a Real Variable, Pergamon Press, Oxford, London, New York, Paris, 1963.

\bibitem{Tim} M. F. Timan, Converse theorems of the constructive theory of functions in the spaces $L_p$ ($1\le p\le\infty$), Mat. sb. \textbf{46} (88) (1) (1958), 125--132.

\bibitem{TB} R. M. Trigub, E. S. Belinsky, Fourier Analysis and
Appoximation of Functions, Kluwer, 2004.




}






\end{thebibliography}
\end{document}